\documentclass[12pt]{amsart}
\usepackage{amsmath,amssymb,amsbsy,amsfonts,latexsym,amsopn,amstext,cite,
                                               amsxtra,euscript,amscd,bm}
\usepackage{url}

\usepackage{comment} 

\usepackage{mathrsfs}

\usepackage{color}
\usepackage[colorlinks,linkcolor=blue,anchorcolor=blue,citecolor=blue,backref=page]{hyperref}
\usepackage{color}
\usepackage{graphics,epsfig}
\usepackage{graphicx}
\usepackage{float}
\usepackage{epstopdf}
\hypersetup{breaklinks=true}

\usepackage[np]{numprint}
\npdecimalsign{\ensuremath{.}}

\def\Th1{\varTheta}

\usepackage{bibentry}

\usepackage[english]{babel}
\usepackage{mathtools}
\usepackage{todonotes}
\usepackage{url}
\usepackage[colorlinks,linkcolor=blue,anchorcolor=blue,citecolor=blue,backref=page]{hyperref}

\usepackage[norefs,nocites]{refcheck}

\begin{document}

\newtheorem{theorem}{Theorem}
\newtheorem{lemma}[theorem]{Lemma}
\newtheorem{claim}[theorem]{Claim}
\newtheorem{cor}[theorem]{Corollary}
\newtheorem{conj}[theorem]{Conjecture}
\newtheorem{prop}[theorem]{Proposition}
\newtheorem{definition}[theorem]{Definition}
\newtheorem{question}[theorem]{Question}
\newtheorem{example}[theorem]{Example}
\newcommand{\hh}{{{\mathrm h}}}
\newtheorem{remark}[theorem]{Remark}

\numberwithin{equation}{section}
\numberwithin{theorem}{section}
\numberwithin{table}{section}
\numberwithin{figure}{section}

\def\sssum{\mathop{\sum\!\sum\!\sum}}
\def\ssum{\mathop{\sum\ldots \sum}}
\def\iint{\mathop{\int\ldots \int}}

\newcommand{\diam}{\operatorname{diam}}

\def\squareforqed{\hbox{\rlap{$\sqcap$}$\sqcup$}}
\def\qed{\ifmmode\squareforqed\else{\unskip\nobreak\hfil
\penalty50\hskip1em \nobreak\hfil\squareforqed
\parfillskip=0pt\finalhyphendemerits=0\endgraf}\fi}%%

%  use the AMS-Euler Fraktur fonts
%%%%%%%%%%%%%%%%%%%%%%%%%%%%%%%%%%
\newfont{\teneufm}{eufm10}
\newfont{\seveneufm}{eufm7}
\newfont{\fiveeufm}{eufm5}
%%%%%%%%%%%%%%%%%%%%%%%%%%%%%%%%%
%
%  allow automatic size selection in math mode
%
%%%%%%%%%%%%%%%%%%%%%%%%%%%%%%%%%
\newfam\eufmfam
     \textfont\eufmfam=\teneufm
\scriptfont\eufmfam=\seveneufm
     \scriptscriptfont\eufmfam=\fiveeufm
%%%%%%%%%%%%%%%%%%%%%%%%%%%%%%%%%
%
%  \frak works on a single symbol at a time...
%
\def\frak#1{{\fam\eufmfam\relax#1}}

\newcommand{\bflambda}{{\boldsymbol{\lambda}}}
\newcommand{\bfmu}{{\boldsymbol{\mu}}}
\newcommand{\bfxi}{{\boldsymbol{\eta}}}
\newcommand{\bfrho}{{\boldsymbol{\rho}}}

\def\eps{\varepsilon}

\def\fK{\mathfrak K}
\def\fT{\mathfrak{T}}
\def\fL{\mathfrak L}
\def\fR{\mathfrak R}

\def\fA{{\mathfrak A}}
\def\fB{{\mathfrak B}}
\def\fC{{\mathfrak C}}
\def\fM{{\mathfrak M}}
\def\fS{{\mathfrak  S}}
\def\fU{{\mathfrak U}}
\def\fW{{\mathfrak W}}

%% Ilya's Macros
\def\D{\Delta}
\def\E{\mathsf{E}}
\def\G{\Gamma}
\def\Gr{{\mathbf G}}
\def\a{\alpha}
\def\d{\delta}
\def\rk{\mathrm{rk}}
%%%
 
\def\vA{\mathbf A}
\def\vB{\mathbf B}

\def\vX{\mathbf X}
\def\vY{\mathbf Y}
 
\def\T {\mathsf {T}}
\def\Tor{\mathsf{T}_d}
\def\Tore{\widetilde{\mathrm{T}}_{d} }

\def\sM {\mathsf {M}}

\def\ss{\mathsf {s}}

\def\Kmnd{\cK_d(m,n)}
\def\Kmnp{\cK_p(m,n)}
\def\Kmnq{\cK_q(m,n)}

\def \balpha{\bm{\alpha}}
\def \bbeta{\bm{\beta}}
\def \bgamma{\bm{\gamma}}
\def \bdelta{\bm{\delta}}
\def \bzeta{\bm{\zeta}}
\def \blambda{\bm{\lambda}}
\def \bchi{\bm{\chi}}
\def \bphi{\bm{\varphi}}
\def \bpsi{\bm{\psi}}
\def \bnu{\bm{\nu}}
\def \bomega{\bm{\omega}}

\def \bell{\bm{\ell}}

\def\eqref#1{(\ref{#1})}

\def\vec#1{\mathbf{#1}}

\newcommand{\abs}[1]{\left| #1 \right|}

\def\Zq{\mathbb{Z}_q}
\def\Zqx{\mathbb{Z}_q^*}
\def\Zd{\mathbb{Z}_d}
\def\Zdx{\mathbb{Z}_d^*}
\def\Zf{\mathbb{Z}_f}
\def\Zfx{\mathbb{Z}_f^*}
\def\Zp{\mathbb{Z}_p}
\def\Zpx{\mathbb{Z}_p^*}
\def\cM{\mathcal M}
\def\cE{\mathcal E}
\def\cH{\mathcal H}

\def\le{\leqslant}

\def\ge{\geqslant}

\def\sfB{\mathsf {B}}
\def\sfC{\mathsf {C}}
\def\sfS{\mathsf {S}}
\def\L{\mathsf {L}}
\def\FF{\mathsf {F}}

\def\sE {\mathscr{E}}
\def\sS {\mathscr{S}}
\def\sL {\mathscr{L}}

%%%%%%%%%%%%%%%%%%%%%%%%%
% Alphabet calligraphie %
%%%%%%%%%%%%%%%%%%%%%%%%%
\def\cA{{\mathcal A}}
\def\cB{{\mathcal B}}
\def\cC{{\mathcal C}}
\def\cD{{\mathcal D}}
\def\cE{{\mathcal E}}
\def\cF{{\mathcal F}}
\def\cG{{\mathcal G}}
\def\cH{{\mathcal H}}
\def\cI{{\mathcal I}}
\def\cJ{{\mathcal J}}
\def\cK{{\mathcal K}}
\def\cL{{\mathcal L}}
\def\cM{{\mathcal M}}
\def\cN{{\mathcal N}}
\def\cO{{\mathcal O}}
\def\cP{{\mathcal P}}
\def\cQ{{\mathcal Q}}
\def\cR{{\mathcal R}}
\def\cS{{\mathcal S}}
\def\cT{{\mathcal T}}
\def\cU{{\mathcal U}}
\def\cV{{\mathcal V}}
\def\cW{{\mathcal W}}
\def\cX{{\mathcal X}}
\def\cY{{\mathcal Y}}
\def\cZ{{\mathcal Z}}
\newcommand{\rmod}[1]{\: \mbox{mod} \: #1}

\def\cg{{\mathcal g}}

\def\vy{\mathbf y}
\def\vr{\mathbf r}
\def\vx{\mathbf x}
\def\va{\mathbf a}
\def\vb{\mathbf b}
\def\vc{\mathbf c}
\def\ve{\mathbf e}
\def\vh{\mathbf h}
\def\vk{\mathbf k}
\def\vm{\mathbf m}
\def\vz{\mathbf z}
\def\vu{\mathbf u}
\def\vv{\mathbf v}

\def\e{{\mathbf{\,e}}}
\def\ep{{\mathbf{\,e}}_p}
\def\eq{{\mathbf{\,e}}_q}

\def\Tr{{\mathrm{Tr}}}
\def\Nm{{\mathrm{Nm}}}

 \def\SS{{\mathbf{S}}}

\def\lcm{{\mathrm{lcm}}}

 \def\0{{\mathbf{0}}}

\def\({\left(}
\def\){\right)}
\def\l|{\left|}
\def\r|{\right|}
\def\fl#1{\left\lfloor#1\right\rfloor}
\def\rf#1{\left\lceil#1\right\rceil}
\def\fl#1{\left\lfloor#1\right\rfloor}
\def\ni#1{\left\lfloor#1\right\rceil}
\def\sumstar#1{\mathop{\sum\vphantom|^{\!\!*}\,}_{#1}}

\def\mand{\qquad \mbox{and} \qquad}

\def\tblue#1{\begin{color}{blue}{{#1}}\end{color}}

 \def\0{{\mathbf{0}}}
  \newcommand{\SL}{\operatorname{SL}}
 \newcommand{\GL}{\operatorname{GL}}
\newcommand{\ord}{\operatorname{ord}}
\newcommand{\Span}{\operatorname{Span}}

\newcommand{\Adj}{\operatorname{Adj}}

%%%%%%%%%%%%%%%%%%%%%%%%%%%%%%%%%%%%%%%%%%%%%%%%%%%%%%%%
%%%%%%%%%%%%%%%%%%%%%%%%%%%%%%%%%%%%%%%%%%%%%%%%%%%%%%%%
%%%%%%%%%%%%%%%%%%%%%%%%%%%%%%%%%%%%%%%%%%%%%%%%%%%%%%%%
%%%%%%%%%%%%%%%%%%%%%%%%%%%%%%%%%%%%%%%%%%%%%%%%%%%%%%%%

%%%%%%%  END OF STANDARD STUFF %%%%%%%%%

%%%%%%%%%%%%%%%%%%%%%%%%%%%%%%%%%%%%%%%%%%%%%%%%%%%%%%%%
%%%%%%%%%%%%%%%%%%%%%%%%%%%%%%%%%%%%%%%%%%%%%%%%%%%%%%%%
%%%%%%%%%%%%%%%%%%%%%%%%%%%%%%%%%%%%%%%%%%%%%%%%%%%%%%%%
%%%%%%%%%%%%%%%%%%%%%%%%%%%%%%%%%%%%%%%%%%%%%%%%%%%%%%%
%%%%%%%%%%%
%%% Spell

\hyphenation{re-pub-lished}

\mathsurround=1pt

\def\bfdefault{b}

\def \F{{\mathbb F}}
\def \K{{\mathbb K}}
\def \N{{\mathbb N}}
\def \Z{{\mathbb Z}}
\def \P{{\mathbb P}}
\def \Q{{\mathbb Q}}
\def \R{{\mathbb R}}
\def \C{{\mathbb C}}
\def\Fp{\F_p}
\def \fp{\Fp^*}

 \def \xbar{\overline x}

\title[Matrices with elements from arbitrary sets]{On the determinants of matrices with elements from arbitrary sets
}

 \author[I. D. Shkredov]{Ilya D. Shkredov}
%\address{ }
%\address{ Department of Mathematics, Purdue University, 150 N. University Street, 
%West Lafayette, IN~47907-2067, USA, and London Institute for Mathematical Sciences, 21 Albemarle St., London W1S 4BS, UK}
\email{ilya.shkredov@gmail.com}
 
 \author[I.~E.~Shparlinski]{Igor E. Shparlinski} 
%  \address{School of Mathematics and Statistics, University of New South Wales.
% Sydney, NSW 2052, Australia}
 \email{igor.shparlinski@unsw.edu.au}

\begin{abstract} Recently there has been several works estimating the number of
$n\times n$ matrices with elements from some finite sets $\cX$ of arithmetic interest  
and of a given determinant. Typically such results are compared 
with the trivial upper bound $O(X^{n^2-1})$, where $X$ is the cardinality of 
$\cX$. Here we show that even for arbitrary sets $\cX\subseteq \R$, 
some recent results from additive combinatorics enable us to obtain a stronger bound with a power saving. 
  \end{abstract}


\maketitle

\tableofcontents

\section{Introduction}

\subsection{Set-up and motivation}

Given a finite set $\cX$ in a  field $K$ and an element $d \in K$, 
we denote by $\cD_n(\cX,d)$ the set of $n \times n$ matrices 
\begin{equation}
\label{eq: Matr X}
\vX=\(x_{i,j}\)_{1\le i,  j\le n}\in \cX^{n\times n}
\end{equation}
such that
\begin{equation}\label{eq: basic det eq}
    \det \vX = d \,, 
\end{equation}
The case of $d=0$, corresponding to singular matrices,  is especially interesting.

We recall that matrices with prescribed determinants and entries from various structural sets have been 
studied a large number of works. In particular, the case of the set $\cX =\{0, 1\}$ and $d=0$ (and the 
dimension $n \to \infty$) has attracted 
a lot of attention, 
see~\cite[Section~7.5]{TaoVu} as well as more recent works~\cite{BVW, CJMS, Tikh, Vu} 
and references therein, 
however many natural questions, such as~\cite[Conjecture~7.24]{TaoVu}, 
remain widely open still. The number of matrices with integer entries from an 
interval $[-H,H]$, $H \to \infty$, (and of fixed dimension $n$) and with a prescribed determinant $d$, 
have been evaluated asymptotically by Duke,  Rudnick and  Sarnak~\cite{DRS},
if $d \ne 0$ is fixed, and  Katznelson~\cite{Kat} when $d=0$ (see also~\cite{BlLu}). 
A uniform (with respect to $d$) 
upper bound on this quantity is given in~\cite{Shp}. 

Motivated by a various arithmetic applications, similar questions have also been studied for matrices with polynomial entries (evaluated at 
integers from $[-H,H]$), see ~\cite{BlLi, KSX, MOS}, while matrices with rational entries have been 
investigated in~\cite{AKOS}. 

To be more precise, we denote by$D_n(\cX,d) = \# \cD_n(\cX,d)$  the number of matrices in $\cD_n(\cX,d)$. 

We recall that the quality of the  bounds from the aforementioned works, 
for special sets $\cX$ of some prescribed arithmetic structure, is usually 
compared with the trivial bound
\begin{equation}
\label{eq: Det Triv}
D_n(\cX,d) = O\(X^{n^2-1}\)\,,
\end{equation}
where  $X = \# \cX$ is the cardinality of $\cX$ and the implied constant depends only on $n$.

It is also interesting to recall that it has been  shown in~\cite[Corollary~3]{Arut} that the number of different determinants $\det \vX$
generated by all  matrices~\eqref{eq: Matr X},  is at least $0.125 X^{0.1 \log n}$, see also~\cite{KPSV, Vinh1, Vinh2}.
 We note that this is established for sets $\cX$ in finite fields (and thus immediately implies this result 
for sets $\cX\subseteq \Q$).  However it is easy to see that  the argument also  works  for sets $\cX \subseteq \R$. 
The proof reduces the problem to some questions of the multiple sum--product phenomenon, see~\cite{TaoVu},  
and this logarithmical growth is a well-known barrier in this area.

\subsection{Main result}

We first observe that in any field $K$,  for $n = 2$ and $d = 0$, the example of the set 
$\cX= \{g^t:~t=1, \ldots, X\}$, for any $g \in K^*$,  shows that~\eqref{eq: Det Triv} is 
tight in this case. 

It is also known that if $K$ is
a finite field $\F_p$ of $p$ elements for a prime $p$ that~\eqref{eq: Det Triv} is tight for $\cX = \F_p$ and any $n$, see~\cite{BreMcK} for 
very general results about distribution of determinants in residue rings.

However, here we show that when $n \ge 3$  
then~\eqref{eq: Det Triv} can be improved for an arbitrary  finite set  $\cX \subseteq \R$.    

Generally, we are interested in bounds which are uniform with respect to $d$, however the case of $d=0$ is
somewhat special and sometimes is more difficult to treat. Hence, we define
\[
D_n^*(\cX) = \sup_{d \in \R^*}  D_n(\cX,d)\mand D_n(\cX) = \sup_{d \in \R}  D_n(\cX,d)\,.
\]

\begin{theorem}\label{thm: det=d}
Let $\cX\subseteq \R$ be an arbitrary set of cardinality   $\# \cX = X$. Then, 
there exists  some absolute constant $\eta>0$ such that 
\begin{itemize}
\item for $n = 3$,   we have 
\[
D_3^*(\cX)   \ll     X^{59/8  -\eta}   \mand  D_3(\cX)  \ll     X^{15/2} \,,
\]
\item for $n \ge 4$, we have 
%IS
\[
D_n^*(\cX)  \ll     X^{n^2 - 13/8  -\eta} \mand
D_n(\cX)  \ll   X^{n^2-2+2/n} \,.
\]  
\end{itemize}
\end{theorem}  

We note that the constant $\eta$ in Theorem~\ref{thm: det=d} originates 
from a result of Rudnev and 
Shkredov~\cite[Theorem~9]{RuSh} (see Lemma~\ref{lem:NX-R} below) and can be explicitly  evaluated. 
However, without an exact knowledge of $\eta$ the bound on $D_n^*(\cX)$ in Theorem~\ref{thm: det=d} 
is meaningful only for $n = 4$ and $n=5$, otherwise the bound on $D_n(\cX) $ is better.

We observe that the dependence on $n\ge 4$ in the bound on $D_n(\cX)$ resembles the 
well--known result of Pila~\cite{Pila} on the number of integer points on a variety, which 
has been improved in a series of work~\cite{BrHBSa,Salb,Ver}.  Despite the similarity of 
the bounds,  our approach does not appeal to these works, and is based on some ideas 
and results from additive combinatorics. 

We also note that our argument works for arbitrary sets from other fields, such as the field of complex 
numbers and finite fields. However, some of the underlying tools are somewhat weaker in 
these fields, which in turn leads to weaker, but still nontrivial, results. For example, in  
Remarks~\ref{rem:Fp 3}  and~\ref{rem:Fp n} 
below, we outline the a path to an analogue of Theorem~\ref{thm: det=d}  over prime finite fields. 

As we have mentioned our approach is based on methods from additive 
combinatorics, in particular on some variants of the Szemer\'edi--Trotter theorem
which may be of independent interest.

\section{Some results from additive combinatorics} 

\subsection{Some multidimensional variants of  the Szemer\'edi--Trotter theorem} 
We start with a bound on incidences between planes  in $\R^k$ and points  from a  Cartesian products of sets from $\R$. 
This method has been developed by Solymosi and de Zeeuw~\cite{Sol_deZ} and by Elekes and T\'oth~\cite{ET}
thus we give just an outline of the proof for the sake of completeness.
We note that we apply Lemma~\ref{lem:Incidence dimk}  below only for $k=3$, however we 
believe that a general statement is of independent interest and deserves to be recorded with a 
full proof. Besides, our proof of Lemma~\ref{lem:scalar_M} re-uses a large portion of the argument 
in the proof of Lemma~\ref{lem:Incidence dimk}.

 We also note that following the established terminology in the area we call a
{\it $d$-flat} any hyperplane of dimension $d$.

\begin{lemma}
\label{lem:Incidence dimk} 
Let $\cP= \cA_1\times \ldots \times \cA_k \subseteq \R^k$ for a fixed integer $k \ge 2$,  be the 
Cartesian product of finite sets  $ \cA_1, \ldots, \cA_k  \subseteq \R$ of 
cardinalities $A_i  = \# \cA_i$, $i =1, \ldots, k$, with $A_1 \ge \ldots \ge A_k$, 
and let $\Pi$ be a set of  $(k-1)$-dimensional hyperplanes in $\R^k$. 
Then 
\begin{align*}
\cI(\cP, \Pi) & \ll    \(\# \Pi \# \cP\)^{k/(k+1)}+    A_k^{-k(k-1)}\(\#\cP\)^k\\
& \qquad +\(A_{k-1} A_k\)^{-1}\# \Pi \# \cP 
+  \# \cP (\# \Pi)^{(k-2)/(k-1)}
\,,
\end{align*}
where  $\cI(\mathcal{P}, \Pi)$ is the number of  incidences between the points 
from  $\cP $ and the planes from $\Pi$. 
\end{lemma} 

\begin{proof} 
We follow the idea of \textit{cells}, introduced in the proofs of~\cite[Theorems~2 and~11]{Sol_deZ}
(see also~\cite[Section~2.3]{Zahl}). 
Namely, first  we fix an integer parameter $r$ with 
\begin{equation}\label{eq:cond_r}
1\le r \le A_k\,.
\end{equation}
to be chosen later. 
Then  slice $\R^k$ by $(k-1)$-dimensional hyperplanes aligned with coordinate 
hyperplanes  into $r^k$ open \textit{cells} such that each cell contains a Cartesian products
$  \cB_1\times  \ldots \times \cB_k$, where 
\[
 \cB_i \subseteq \cA_i \mand  \# \cB_i   = \frac{1}{r} \#\cA_i+ O(1), \qquad i =1, \ldots, k\,.
 \]
 
 %IS
 Observe that any hyperplane  intersects at most $kr^{k-1}$ cells, 
 see  the proofs of~\cite[Theorems~2 and~11]{Sol_deZ}. 
 Indeed, in each  of $k$ orthogonal directions in $\R^k$, 
 we can consider the whole configuration as a union of $r^{k-1}$ "slices" 
 (which are only one cell thick in one direction). 
 When a  $(k-1)$-dimensional hyperplane intersects one cell, it removes all cells in 
 at  least on whole slice (in at least one direction) from further considerations. This gives 
 the above bound on the number of intersected cells.

%IS
Further, incidences $ \cI(\cP,\Pi)$ 
%Hyperplanes  from $\Pi$ 
can be partitioned 
into the following three classes:
\begin{itemize}
\item $\cI_1$, which contains  $\(P,\,\pi\)  \in \cP \times  \Pi$ with $P \in \pi$ and  such that  $\pi$ has at most $k-1$ points in the cell, containing $P$;
\item $\cI_2$, which contains   $\(P,\,\pi\)   \in \cP \times  \Pi$  with $P \in \pi$ and  $\pi$ is uniquely defined by some  $k$-tuple of  distinct points from the  cell,   containing $P$;
\item  $\cI_3$, which contains   the  rest of $\(P,\,\pi\)  \in \cP \times  \Pi$  with $P \in \pi$. 
\end{itemize}

%% 
%%By the argument above  
%%\] 
%%%\cI(\cP,\Pi_1) 
%%\cI_1
%%\ll    r^{k-1} \# \Pi \,.
%%\] 

For each cell $\fC$ the contribution to $\cI_1$ can be estimated as 
\[
 \sum_{\pi \in \Pi} \sum_{\substack{P \in \fC \cap \cP\\ \(P,\,\pi\) \in \cI_1}} 1 
\le (k-1)  \# \Pi \,.
\]
Summing over all $O\(r^{k-1}\)$ cells  
we obtain 
\[
\# \cI_1 \le \sum_{\fC} \sum_{\pi \in \Pi} \sum_{\substack{P \in \fC \cap \cP\\ \(P,\,\pi\) \in \cI_1}} 1 = 
 \sum_{\pi \in \Pi} \sum_{\fC:~\fC \cap \Pi \ne \varnothing}  \sum_{\substack{p \in \fC \cap \cP\\ \(P,\,\pi\) \in \cI_1}} 1 \ll    r^{k-1} \# \Pi \,.
\]
%%\ll    r^{k-1} \# \Pi \,.

%IS
Also, every cell, contains $O\(\(\#\cP\)/r^k\)$ points from $\cP$ 
and hence generates at most  $O\(\(\(\#\cP\)/r^k\)^k\)$ hyperplanes from $\cI_2$. 
Hence, 
\begin{equation}\label{f:I_2}
\# \cI_2
 \ll  \( \frac{\#\cP}{r^k} \)^k r^k =   \frac{\(\#\cP\)^k }{r^{k(k-1)}} \,.
\end{equation}

 Finally, let us 
 calculate the quantity $\#\cI_3$, that is the number of pairs $\{ (P,\pi) ~:~ P\in \cP,\, \pi \in \Pi\} \in \cI_3$ (for $k=2$ one has $\cI_3=0$ and hence below we assume that $k>2$). 
By the definition of $\cI_3$ the intersection $\cP \cap \pi$ belonging to a cell 
 has rank $\ell \le k-1$, in other words, it belongs to a $k-2$-flat. 
First suppose that {\it all} points of the intersection $\cP \cap \pi$ belong at most two $k-2$-flats and denote the number of such incidences as $\overline \cI_3$ and let the rest be $\widetilde\cI_3$.

 Let $\ell$ be the largest integer for which there are $\ell$ linearly 
independent points  $P_i = \(a_{i,1}, \ldots, a_{i,k}\)$, $ i =1, \ldots, \ell-1$, which we represent as columns vectors.
 Then for any other point 
\[
P = \(b_{1}, \ldots, b_{k}\)^t\in \cP\cap \pi\,, 
\]
(where $\vb^t$ denotes the transposition of a vector $\vb$) 
using that $P \in \pi$ we express one component of $P$ via the remaining components. 
After this the condition that the matrix 
$\(P_1 \vert \ldots \vert P_k\vert P\)$ with columns formed by the above points
is of rank $\ell$ gives rise a nontrivial equation on these remaining $k-1$ components. 
Therefore, for any 
such a plane $\pi$  
may contain at most $T = O\(\(\#\cP\)/ \(A_{k-1}A_k\)\)$ points 
from $\cP$. Hence, 
\[
 \#\overline\cI_3
 \ll   \frac{\# \Pi \#\cP}{A_{k-1}A_k} \,.
\]
Also, notice that if our plane contains at most $2T$ points, then the contribution of such planes into $\cI(\cP, \Pi)$  is $O\(\(\#\Pi \#\cP\)/ \(A_{k-1}A_k\)\)$ and hence has the same order as $ \#\overline\cI_3$.

It remains to estimate the value $\widetilde\cI_3$. 
For $k=3$ it has been shown in~\cite{Rud} that the contribution of  $\widetilde\cI_3$ is negligible (under the assumption $\# \Pi \ge \# \cP$). In the case of general $k$ we appeal to~\cite{Lund} which develops the original paper of Elekes and T\'oth~\cite{ET}. 
Recall, that for given $\a \in (0,1)$ a $(k-1)$-flat $\pi$ is $\a$-nondegenerate if at most $\a \#\(\cP\cap \pi\)$ points of $\cP$ lie on any $(k-2)$-flat contained in $\pi$.  
By our construction any plane in  $\widetilde\cI_3$ is $1/2$-nondegenerate and thus by the upper bound for the number of $\a$-nondegenerate incidences, see ~\cite[Theorem~4]{Lund}, one has 
\begin{equation}\label{f:I''_3}
	\#\widetilde\cI_3 \ll   \(\# \Pi \# \cP\)^{k/(k+1)} + \#\Pi + \# \cP (\# \Pi)^{(k-2)/(k-1)} \,.
\end{equation}

%IS
Combining the above bounds, we derive
\begin{equation}\label{eq:Prelim}
\begin{split}
  \cI (\cP, \Pi) & 
  %\le \cI(\cP,\Pi_1) + \cI(\cP,\Pi_2) + \cI(\cP,\Pi_3)\\
  =  \# \cI_1 +  \# \cI_2 +  \# \cI_3 \\
  & =    \# \cI_1 +  \# \cI_2 +  \# \overline\cI_3 + \#\widetilde\cI_3\\ 
 &  \ll  r^{k-1} \# \Pi +    \frac{\(\#\cP\)^k }{r^{k(k-1)}} +   \frac{\# \Pi \#\cP}{A_{k-1}A_k} 
 +  \# \cP (\# \Pi)^{(k-2)/(k-1)} 
 \,.
\end{split}
\end{equation}

To optimize these bounds, we choose 
\begin{equation}\label{def:r}
r=\min\left\{\max\left\{ \fl{\(\(\#\cP\)^k/\# \Pi\)^{1/(k^2-1)}}, 1\right\},  A_k\right\}\,, 
\end{equation}
which  satisfies~\eqref{eq:cond_r} and thus is 
an admissible choice.  This leads to the bound
\begin{align*}
  \cI (\cP, \Pi) &   \ll  \(\# \Pi \# \cP\)^{k/(k+1)}  + \# \Pi  +    \frac{\(\#\cP\)^k }{A_k^{k(k-1)}}  \\
  & \qquad \qquad \qquad \qquad \qquad +   \frac{\# \Pi \#\cP}{A_{k-1}A_k} 
  +  \# \cP (\# \Pi)^{(k-2)/(k-1)}  \\
   &   \ll  \(\# \Pi \# \cP\)^{k/(k+1)}  +    \frac{\(\#\cP\)^k }{A_k^{k(k-1)}} +    \frac{\# \Pi \#\cP}{A_{k-1}A_k} 
   +  \# \cP (\# \Pi)^{(k-2)/(k-1)}
\end{align*}
and  completes the proof. 
\end{proof}

Finally, we also need the following result. 

\begin{lemma}
\label{lem:scalar_M}
    Let $k\ge 2$ be an integer and $M$ be a non-singular $k\times k$ matrix over $\R$ and $\omega \neq 0$ be any real number. 
    Also, let $\cB, \cC   \subseteq \R$ 
    %and  $\cC \subseteq \R^k$ 
    be some finite sets 
    or cardinalities $B$ and $C$, respectively. 
    Suppose that $B\le C$. 
    Then the equation 
\[
    \langle M \vb, \vc\rangle = \omega \,, \qquad  \vb\in \cB^k\,, \  \vc \in \cC^k \,,
\]
defined by the inner product,   has  
$O \( \(B^k  C^k\)^{k/(k+1)} + B^{k} 
	+ B^{k-2} C^k \)$ solutions. 
\end{lemma}

\begin{proof}  We write our equation  as
\[
 b_1 (M\vc)_1 + \ldots + b_k   (M\vc)_k = \omega \,,
\]
where $b_1, \ldots, b_k \in \cB$,  $\vc = (c_1,  \ldots, c_k)\in \cC^k$ and $(M\vc)_j$ the $j$-th component 
of the vectors $M\vc$.  
    
 We interpret the number of solutions to the above equation as the number of incidences between points $\cB^k$ and hyperplanes, indexed by $\vc\in \cC^k$, which are all distinct since $\omega \ne 0$. 
Applying  Lemma~\ref{lem:Incidence dimk}, we obtain the  desired result, except for  the term $\# \cP (\# \Pi)^{(k-2)/(k-1)}$, 
which comes from the bound on $\# \widetilde\cI_3$, see the estimate~\eqref{f:I''_3}.

To improve this term, we notice that our set of planes has the special form, namely, it is an affine image 
(induced by a non-singular matrix $M$) of $\cC^k$. It means, in particular, that there are at most $C^{k-1-d}$ hyperplanes, containing a flat of dimension $d<k-1$. 
Similarly, the number of points in the intersection of $d$-flat with each cell is at most $(B/r)^d$. 
Finally, notice that  we do not need to 
take into account  the incidences that 
are already included in $\cI_2$ and hence we can suppose that the intersection of any cell 
with the remaining hyperplanes is a flat of dimension $1\le d \le k-2$.

Therefore, we get 
\begin{align*}
\widetilde\cI_3 & \ll r^k \sum_{d=1}^{k-2}  (B/r)^d C^{k-1-d} = r^k C^{k-1} \sum_{d=1}^{k-2} \left( \frac{B}{rC} \right)^d\\
&\ll   r^{k-1} B C^{k-2} + r^2 B^{k-2} C \,.
\end{align*}

Using these facts, and calculations as in~\eqref{f:I_2} instead of~\eqref{eq:Prelim}
we obtain 
\begin{align*}
	 \cI (\cP, \Pi) & \ll  r^{k-1} C^k + \frac{B^{k^2}}{r^{k(k-1)}} \  + B^k + B^{k-2} C^k  \\
	 & \qquad \qquad  \qquad \qquad  \qquad \qquad +  r^{k-1} B C^{k-2} + r^2 B^{k-2} C\,.
\end{align*}
By our assumption $B \le C$ we have $r^{k-1} B C^{k-2}\ll r^{k-1} C^k$. 
Furthermore, since $r \le B$ and $k\ge 3$, we also have 
\[
r^2 B^{k-2} C\le  B^{k} C \le  B^{k-2} C^k\,
\] 

Hence the above bound simplifies as 
\[
 \cI (\cP, \Pi)  \ll  r^{k-1} C^k + \frac{B^{k^2}}{r^{k(k-1)}} \  + B^k + B^{k-2} C^k\,.
\]
Choosing 
%$r^{k-1}=B^{k^2}$, we get 
the parameter $r$ as in~\eqref{def:r}, that is, 
\[
r=\ \max\left\{ \fl{\(B^{k^2}/C^k\)^{1/(k^2-1)}}, 1\right\}\,, 
\] 
we derive 
\[
	 \cI (\cP, \Pi)  \ll   \(B^k  C^k\)^{k/(k+1)} + B^{k} 
	+ B^{k-2} C^k\,, 
\]
as required. 
\end{proof}

 \begin{remark}
\label{rem:B vs C} 
It is easy to see that the condition $B \le C$ in  Lemma~\ref{lem:scalar_M} can be significantly 
relaxed, but it is quite sufficient for our purpose. 
\end{remark}

\subsection{Some low-dimensional improvements} 

Given a  finite set $\cU  \subseteq \R$, we use   $N(\cU)$ 
to denote the number of solutions to the equation 
\[
v_1(u_1-w_1) = v_2 (u_2-w_2)  
\]
with $u_\nu, v_\nu,w_\nu \in \cU$, $\nu =1,2$.

We recall the following bound given by Rudnev and 
Shkredov~\cite[Theorem~9]{RuSh},  which applies to sets 
of real numbers.

\begin{lemma}\label{lem:NX-R} 
There is an absolute constant $\vartheta >0$, such that 
for any set 
$\cU\subseteq  \R$  of cardinality  
$\# \cU = U$, we have
\[
N(\cU) \ll U^{9/2 - \vartheta}\,. 
\] 
 \end{lemma}

 Given a  finite set $\cU  \subseteq \R$, we use   $T(\cU)$ 
to denote the number of solutions to the equation 
\[
v_1u_1 + v_2 u_2 =  x_1y_1+x_2y_2
\]
with $u_\nu, v_\nu,x_\nu, y_u \in \cU$, $\nu =1,2$. 

Our next result can be considerede as a low-dimensional 
improvement of Lemma~\ref{lem:scalar_M} in some 
special cases. 

\begin{lemma}\label{lem:TX-R} 
There is an absolute constant $\vartheta >0$, such that 
for any set 
$\cU\subseteq  \R$  of cardinality  
$\# \cU = U$, we have
\[
T(\cU) \ll U^{13/2 - \vartheta}\,. 
\] 
 \end{lemma}
 
 \begin{proof} Let $R(t)$ be the number of solutions to the equation 
\[
u_1v_1 + u_2 v_2 = t, \qquad u_1,u_2, v_1, v_2  \in \cU\,,  
\]
 where $t$ runs through the set $\cT$ of 
 all $t\in \R$, which admit such a representation. 
 
 Hence, by the Cauchy inequality, 
 \begin{align*}
T(\cU)  & \ll   \sum_{t \in \cT} R(t)^2 =
  \sum_{t \in \cT} \(\sum_{\substack{u_1v_1 + u_2 v_2 = t\\u_1,u_2, v_1, v_2  \in \cU}} 1\)^2\\
   & \le U^2  \sum_{t \in \cT}\, \sum_{u_1,u_2  \in \cU}
    \(\sum_{\substack{u_1v_1 + u_2 v_2 = t\\  v_1, v_2  \in \cU}} 1\)^2\\
 & = U^2  \sum_{u_1,u_2  \in \cU}\,
 \sum_{\substack{u_1v_1 + u_2 v_2 = u_1v_3 + u_2 v_4 \\  v_1, v_2, v_3, v_4  \in \cU}} 1
 = U^2 N(\cU)\,. 
\end{align*}
 Applying Lemma~\ref{lem:NX-R}, we complete the proof. 
\end{proof}

  \section{Proof of Theorem~\ref{thm: det=d}} 
  
\subsection{The case of $n=3$}
\label{subsec:n=3}

Let $N =  D_3(\cX,d)$  for some $d\in \R^*$ with $D_3(\cX) = D_3(\cX,d)$ (clearly $D_n(\cX,d)$ is supported on a finite 
set, hence such $d$ always exists for any $n$).

We need to estimate the number of solutions to the equation 
\[
\det \vX = d\,,
\]
where 
\begin{equation}\label{eq: M 3*3}
\vX = \begin{pmatrix}
x_1 & x_2 & x_3 \\ 
y_1 & y_2& y_3 \\
z_1 & x_2 & z_3
\end{pmatrix}  , \qquad x_i, y_i, z_i \in \cX  \,. 
\end{equation}

We see that $N$ is the number of solutions to the equation
\begin{equation}\label{eq:n=3,d}
    x_1(y_2z_3 - y_3 z_2) - x_2 (y_1 z_3 - y_3z_1) + x_3 (y_1 z_2 - y_2z_1) = d \,, 
\end{equation}
which  we can view as the question on the number of points/planes incidences between the points 
 from $\cP = \cX \times \cX \times \cX$ and the planes from the set $\Pi$ of planes with  projective coordinates
  \begin{equation}\label{eq:planes}
\(\(y_2z_3 - y_3 z_2\) : - \(y_1 z_3 - y_3z_1\):  \(y_1 z_2 - y_2z_1\) : d\)
 \end{equation}
 
 Observe that since $d\ne 0$, the planes with   projective coordinates~\eqref{eq:planes} are distinct  
 if and only if the corresponding  affine points 
\[
 \(\(y_2z_3 - y_3 z_2\),  - \(y_1 z_3 - y_3z_1\),  \(y_1 z_2 - y_2z_1\) \)\in \R^3
\] 
are distinct.
 
Next, let $\mu(a,b,c)$ be the multiplicity of appearance of the projective point $(a:b:c:d)$ in~\eqref{eq:planes}.
 
 For each $w$, we  form the set $\Pi_w$ of the planes which are given by projective coordinates~\eqref{eq:planes}
 with  
\[
w \le \mu\(\(y_2z_3 - y_3 z_2\),  - \(y_1 z_3 - y_3z_1\),  \(y_1 z_2 - y_2z_1\) \) < 2w\,.
\]

So,  by the Dirichlet principle, we can find some $w>0$ and a set $\Pi_w$ such that 
\[
    N   \ll Lw \cI(\mathcal{P}, \Pi_w) \,.
\]  
where $L = \log X$. 

Applying Lemma~\ref{lem:Incidence dimk} with $k=3$ and $\cA_1 = \cA_2 = \cA_3 = \cX$, we obtain 
\begin{align*}
 N  & \ll     w L \(\(\# \Pi_w  X^3\)^{3/4} + X^{-6} \(X^3\)^3  +\# \Pi_w X + X^3  (\# \Pi_w)^{1/2}\)\\
   & =
    LX^{9/4} \(w\#\Pi_w\)^{1/2} \(w^2 \# \Pi_w\)^{1/4}+      wL\# \Pi_w X +
    LX^3 (w^2 \# \Pi_w)^{1/2}  \\
  & \ll     LX^{9/4+3}  \(w^2 \# \Pi_w\)^{1/4}    + LX^7 
\end{align*}
(since we have the trivial bounds $w \# \Pi_w\le X^6$ and $w^2 \# \Pi_w\le X^9$).  
%(since we have the trivial bound $w \# \Pi_w\le X^6$).  

We now observe that 
\[
w^2 \# \Pi_w \le \sum_{\substack{y_i,z_i \in \cX\\i =1,2,3}}
 \mu\(\(y_2z_3 - y_3 z_2\),  - \(y_1 z_3 - y_3z_1\),  \(y_1 z_2 - y_2z_1\) \)^2\,. 
\]
Therefore, 
\begin{equation}\label{eqf:n=3,d}
 N   \ll L X^{9/4+3} \E^{1/4}_*  +  LX^7  = L X^{21/4} \E^{1/4}_* + LX^7  \,,
\end{equation}
where $\E_*$ is the number of the solutions to the system 
\begin{equation}\label{eq:MatrEq}
 \begin{pmatrix}
u_1 & 0 & -u_3 \\ 
0 & u_3 & -u_2 \\
u_2 & -u_1 & 0
\end{pmatrix}
 \begin{pmatrix}
v_3 \\ 
v_2 \\
v_1
\end{pmatrix}
= 
 \begin{pmatrix}
y_1 & 0 & -y_3 \\ 
0 & y_3 & -y_2 \\
y_2 & -y_1 & 0
\end{pmatrix}
 \begin{pmatrix}
z_3 \\ 
z_2 \\
z_1
\end{pmatrix}
\end{equation}
in variables $u_i, v_i, y_i, z_i \in \cX$, $ i=1,2,3$.

We fix $u_2, y_2\in \mathcal{X}$ (in $X^2$ ways). 
We also observe  that by Lemma~\ref{lem:TX-R},  the equation
\[
u_1v_3  -u_3  v_1= y_1 z_3  -y_3 z_1
\]
implied by~\eqref{eq:MatrEq},  has at most $O(X^{13/2-\vartheta})$ solutions, where $\vartheta >0$ is an 
absolute constant. Thus, $\E_* \ll X^{13/2+2-\vartheta} = X^{17/2-\vartheta}$ and therefore~\eqref{eqf:n=3,d} 
implies  
\[
    N \ll L X^{15/2 - 1/8-\vartheta/4} \,, 
\]
which gives the desired result for $ D_3^*(\cX)$.   

Now let $d=0$ and let $N_0 =   D_3(\cX,0)$, which is the only remaining case to investigate.

In this case the equation~\eqref{eq:n=3,d} becomes
\begin{equation}\label{eq:n=3,d=0-}
    x_1(y_2z_3 - y_3 z_2) - x_2 (y_1 z_3 - y_3z_1) + x_3 (y_1 z_2 - y_2z_1) = 0 \,.
\end{equation}
    We can assume that all variables are non--zero because otherwise  we get at most $O(X^7)$ solutions.
Then we fix  $x_3, y_3, z_3 \in \cX$, which is possible in $X^3$ way and for convenience 
we rename them as 
\[
\(x_3, y_3, z_3\) = \(a,b,c\)\, .
\]

First, dividing by $c$,   we rewrite the equation~\eqref{eq:n=3,d=0-} as 
 \[
x_1\(y_2-  bz_2^*\) - x_2 \(y_1 -  bz_1^* \) + a \(y_1 z_2^*-    y_2z_1 ^*\) = 0 \,,
\]
with variables 
\[
x_1,x_2,  y_1, y_2 \in \cX \mand  z_1^*, z_2^* \in c^{-1} \cX\,, 
\]
where, as usual, $c^{-1} \cX = \{c^{-1} x:~x \in \cX\}$. 

Next, we divide the previous equation by $b$,  getting
 \[
x_1\(y_2^*-  z_2^*\) - x_2 \(y_1^* -  z_1^* \) + a \(y_1^* z_2^*-    y_2^*z_1 ^*\) = 0 
\]
with variables 
\[
x_1,x_2 \in  \cX, \qquad  y_1^*, y_2^* \in b^{-1} \cX, \qquad  z_1^*, z_2^* \in c^{-1} \cX\,.
\]

Finally, dividing this equation by $a$ we arrive to 
 \[
x_1^*\(y_2^*-  z_2^*\) - x_2^* \(y_1^* -  z_1^* \) +  y_1^* z_2^*-    y_2^*z_1 ^* = 0  
\]
with variables 
\[
x_1^*,x_2^*  \in a^{-1} \cX, \qquad  y_1^*, y_2^* \in b^{-1} \cX, \qquad  z_1^*, z_2^* \in c^{-1} \cX\,.
\]
Putting all variables in one set 
\[
\cU = \(a^{-1} \cX\) \cup\(b^{-1} \cX\) \cup\(c^{-1} \cX\)\,,
\]
of cardinality $\# \cU \le 3X$,  we arrive to the problem of bounding solutions to the equation 
\begin{equation}\label{eq:n=3,d=0}
u_1\(v_2-  w_2\) - u_2 \(v_1 -  w_1 \) +  v_1 w_2 -    v_2 w_1 = 0 \,, 
\end{equation}
with 
\[
u_1,u_2,  v_1, v_2, w_1, w_2 \in \cU\,.
\]
It is now convenient to rename $( v_1, w_1, w_2)$   as $(a,b,c)$ and 
rename $(u_1, u_2, v_2)$ as $(r,s,t)$.

With these new variables, we consider the  equation~\eqref{eq:n=3,d=0} as a question about the incidence of points
$(r,s,t) \in \cU^3$   and the family  of 
quadratic affine curves  
\[
\cC_{\va}: ~   R(T-c) - S(a-b) +ac -    Tb = 0, \quad  \va = (a,b,c)\in \cU^3 \,,
\]
which have to contain the above point after we specialize
\[(R,S,T) \to (r,s,t)\,.
\] 
We now rewrite the defining equation of the curves $\cC_{a,b,c}$ as
\begin{equation}\label{eq:curve}
\cC_{\va}: ~   RT- \cL_\va(R,S,T) = 0, \quad  \va = (a,b,c)\in \cU^3 \,,
\end{equation}
where $\cL_\va(R,S,T)  = cR  + S(a-b) + bT - ac \in \C[R,S,T]$. 

In order to apply~\cite[Theorem~1.2]{ Zahl}, we have to check that any three distinct curves
from the family of curves~\eqref{eq:curve}, intersect in   $O(1)$ points, 
and also that  any three points determine $O(1)$ curves. 

We first observe that distinct triples $\va$ correspond to distinct curves $\cC_{\va}$.

Next, we note that any two distinct linear polynomials $\cL_{\va_1}(R,S,T)$ and 
 $\cL_{\va_2}(R,S,T)$  of the above form are linearly independent. This in turn implies that 
 for any three pairwise distinct linear polynomials $\cL_{\va_i}(R,S,T)$, $i=1,2,3$,   at least two out the 
 following three differences
 \[
 \cM_{1,2} = \cL_{\va_1} - \cL_{\va_2} , \quad  \cM_{2,3} = \cL_{\va_2} - \cL_{\va_3} , 
 \quad   \cM_{3,1} = \cL_{\va_3} - \cL_{\va_1}\,, 
 \]
 are linearly independent. Assuming that, say,  $ \cM_{1,2}$ and $\cM_{2,3}$
 are linearly independent,    from the definition of the curves 
 $\cC_{\va}$ in~\eqref{eq:curve}, that we have a system of equation 
\[
 \cM_{1,2}(r,s,t) = \cM_{2,3} (r,s,t) = 0\,.
 \]
 of rank two. Hence we can express two of three variables via the third one and then 
 substitute in the equation $rt -  \cL_{\va_1} (r,s,t) =0$,  which becomes a nontrivial quadratic 
 equation in this third variable. This shows that any three distinct curves~\eqref{eq:curve} 
 have at most  two  common points. 
 
 An essentially identical argument shows that any three distinct point can simultaneously  
belong to at most two curves~\eqref{eq:curve}. 

We are now able to apply~\cite[Theorem~1.2]{Zahl} (with $k=3$, $m = n= U^3$),
getting that the number of the above incidences is  
\[
O\(\(U^3\)^{6/8} \(U^3\)^{6/8} + U^3\) = O(U^{9/2}) =  O(X^{9/2}) 
\]
and hence, recalling that the number of choices for $x_3, y_3, z_3 \in \cX$ is  $X^3$,
we obtain  $N_0\ll X^3 X^{9/2} = X^{15/2}$, which gives the desired bound on $D_3(\cX)$.

\begin{remark}
\label{rem:Fp 3} 
Now, let   $D_{n,p}(\cX,d)$ be an analogue  of $D_n(\cX,d)$ defined for a 
set $\cX$ from the field $\F_p$ of $p$ elements, where $p$ is 
a sufficiently large prime. 
Using  a result of Rudnev~\cite[Theorem~3]{Rud} and our calculations from the beginning 
of the proof of Theorem~\ref{thm: det=d},  we see that 
for $d \in \F_p^*$ we have
\begin{align*}
D_{3,p}(\cX,d) & \ll wL \( \frac{\# \Pi_w X^3}{p}+ \# \Pi_w X^{3/2} + X \# \Pi_w + X^{9/2} \)\\
&  \ll \frac{L X^9}{p} + L X^{8-1/2}\, ,
\end{align*}
where, as before,  $L = \log X$.
Thus for $X\le p^{2/3}$ we have the saving  of $X^{1/2}$ compared to the  
ananlogue of the trivial bound~\eqref{eq: Det Triv}, while for $X > p^{2/3}$ the first term dominates 
and it is certainly asymptotically optimal. Furthermore, if $d=0$, then one can fix two more variables in~\eqref{eq:curve}, 
for example, $T$ and $a$, and then apply the points/lines incidences of Stevens and de Zeeuw~\cite[Lemma~5]{SZ} 
to get a slightly weaker bound $D_{3,p}(\cX,0)  \ll X^9/p + X^{8-1/4}$. 
\end{remark}

\subsection{The case of $n \ge 4$: First bound}
\label{sec: large n -1} 
We consider  a $3\times 3$ minor as in~\eqref{eq: M 3*3}  and denote it by $\vY$, keeping the 
name $\vX$ for the whole  $n\times n$ matrix as in~\eqref{eq: Matr X}. 

Further, we consider the complementary $(n-3) \times (n-3)$ minor $\vY^c$ using the remaining columns and rows. 
It is easy to see that the contribution to $D_n(\cX,d)$ from matrices $\vX$ with $\det \vY^c = 0$ is $O\(X^{n-2}\)$
as in this case we get two independent equations). Indeed, it is obvious for $d\neq 0$ but for $d=0$ the polynomial 
$\det \vX \in \R[\vX]$ cannot be divisible by the polynomial $\det \vY^c  = 0$
 (consider any $n$ linearly independent vectors with fixed components from 
 %IS
 %%the matrix $\det \vY^c $,
 the matrix $\vY^c $,
 the choice of such 
 vectors  is possible for $n \ge 4$).

Hence we now assume that   the remaining $n^2-9$ elements outside of $\vY$  in such a way that $\det \vY^c \neq 0$. 

Applying Gauss elimination we can make $\vY^c$ equals the identity matrix $I_{n-3}$, 
this changes the value of $d$ in the equation~\eqref{eq: Matr X}  to a non-zero multiple, and hence if $d \ne 0$ 
this new value is also non-zero. 

After this, via elementary row and column transformations we can make all remaining 
$n^2 - (n-3)^2 - 9 = 6(n-3)$ elements  of $\vX$ (that is, excluding elements of $\vY$ and $\vY^c$) equal zero
(and also leading to an equation of the type~\eqref{eq: basic det eq} with some different, but still non-zero, value 
of $d$. 
Observe that this also induces some additive shifts $\xi_{i,j}$, $1 \le i,j\le 3$ of the elements of $\vY$. 
Now applying our bounds for the case $n = 3$ with the new set 
\[
\widetilde \cX = \bigcup_{1 \le i,j\le 3} \( \cX + \xi_{i,j}\)\,.
\]
we conclude that $D_n^*(\cX) \ll X^{n^2-9}  D_3^*(\widetilde \cX ) + X^{n^2-2} \ll X^{n^2-13/8-\eta}$. 

\begin{remark}
The above argument also works for $d=0$ and leads to the inequality 
 $D_n(\cX) \ll X^{n^2-9}  D_3(\widetilde \cX ) + X^{n^2-2} \ll X^{n^2-3/2} $, 
 which however is always weaker than our next bound in Section~\ref{sec: large n -2}. 
\end{remark} 

\begin{remark}
\label{rem:Fp n} 
This argument, combined with the bounds  on $D_{3,p}(\cX,d)$ indicated  Remark~\ref{rem:Fp 3}  can also be used to estimate 
$D_{n,p}(\cX,d)$ for any $n \ge 3$. 
\end{remark}   
 
\subsection{The case of $n \ge 4$: Second bound}
\label{sec: large n -2}

As before, let $N =  D_n(\cX,d)$  for some $d\in \R^*$ with $D_n(\cX) = D_n(\cX,d)$.

We represent $\vX \in \cD_n(\cX,d)$ in the form
  \begin{equation}
  \label{eq:XYxy}
\vX=\begin{pmatrix}  \vY & \vy^t\\ \vz &  x \end{pmatrix}, \qquad x \in \cX, \  \vy,  \vz \in \cX^{n-1}, \
\vY \in \cX^{(n-1) \times (n-1)}\,.
  \end{equation}

The contribution to $N$ from matrices~\eqref{eq:XYxy} with $x = 0$ is clearly $O(X^{n^2-2})$.
Hence, 
  \begin{equation}
  \label{eq:NNN}
N = N_0 + N_*  + O(X^{n^2-2})\,, 
  \end{equation}
  where $N_0$ is the contribution from from matrices~\eqref{eq:XYxy} with   $\det \vY = 0$ 
  and $N_*$ is the contribution from other matrices~\eqref{eq:XYxy}. 
  
To estimate $N_0$, we note that there are most $O\(X^{(n-1)n -2}\)$ choices for  the 
 $(n-1)\times n$-matrix $\(\vY|  \vy^t\)$ for which it is or rank at most $n-2$. 
 Otherwise we have two equations $\det \vY =0$ and then when $\vY$ and $\vy$ are fixed
 in at most $O\(X^{(n-1)^2 - 1 + n -1}\) = O\(X^{(n-1)n-1}\)$ ways and with  $\(\vY|  \vy^t\)$
 of rank $n-1$, we obtain a nontrivial linear equation for $\vz$. 
 Hence 
 \begin{equation}
  \label{eq:N0}
 N_0  \ll X^{(n-1)n -1 + n-1} = X^{n^2-2}\,.
 \end{equation}

To estimate $N_*$, we  fix one of at most distinct $X^{(n-1)^2}$ matrices 
$\vY$ with  $\det \vY \neq 0$. Then, by the  {\it Schur determinant identity\/}, noticing that $\vz  \vY^{-1} \vy^t$ is a scalar, we have 
\[
\det \vX = \det \vY \cdot \det\(x -  \vz  \vY^{-1} \vy^t\)
=  x \det \vY - \vz \Adj(\vY) \vy^t\,, 
\]
where $ \Adj(\vY) $ is the adjoint matrix of $\vY$, that is, 
\[
\vY  \Adj(\vY) = \det \vY \cdot I_{n-1}\,.
\] 

 We notice that there is a unique choice of $x$ with 
$x \det \vY= d$, leading to the nontrivial equation $\vz \Adj(\vY) \vy^t = 0$ 
which can be satisfied by at most $X^{2(n-1)-1} = X^{2n-3}$ vectors $  \vy,  \vz \in \cX^{n-1}$. 
Thus the 
total number of such matrices is at most $X^{(n-1)^2-(2n-3)} = X^{n^2-2}$. 

Next, for other choices of $x$ and $\vY$ with  $\det \vY \neq 0$, since  
$\Adj(\vY)$ is non-singular and we can apply Lemma~\ref{lem:scalar_M}
with $M=\Adj(\vY)$, $k=n-1$, $B = X$, $C = X^{n-1}$ and $\omega = d- x \det \vY \ne 0$. Hence 
we have 
\[
O\( \(X^{n-1} \cdot  X^{n-1} \)^{(n-1)/n} + 
%X^{(n-1)(2n-5)/(n-2)} 
X^{n-1}
+ X^{n-3} \cdot X^{n-1} \) 
= O\(X^{2n -4 +2 /n}\)\,
\]
choices for  $(\vy, \vz) \in \cX^{n-1} \times \cX^{n-1}$. Therefore, 
 \begin{equation}
  \label{eq:N*}
 N_*  \ll X^{(n-1)^2 + 1 + 2n -4 +2 /n} = X^{n^2-2+2 /n}\,.
 \end{equation}
Substituting~\eqref{eq:N0} and~\eqref{eq:N*} in~\eqref{eq:NNN}
implies that 
\[D_n(\cX) = N \ll X^{n^2-2+2 /n}
\]
and thus completes the proof.

 \section{Comments} 
 
 \subsection{Conjectured estimate}

In case $\cX = \{1, \ldots, X\}$ the results of~\cite{DRS} and~\cite{Kat} show that for this set
\[
D_n(\cX)  \ll  X^{n^2-n}\log X \mand 
D_n^*(\cX)  \ll     X^{n^2 - n}\,.
\] 
However, for sparse sets one can get a lower  bound, which shows that  the 
naive guess that $D_n(\cX) \le  X^{n^2-n+o(1)}$ for any set $\cX \in \R$ is false. 
 For example, if $\cX = \{\gamma^x:~ x =1, \ldots, X\}$
for some $\gamma \ne 0, \pm 1$, then it is easy to see that  $D_n (\cX, 0) \gg X^{n^2-n+1}$. 
In particular, partially motivated by work of Arutyunyan~\cite{Arut} on the number of distinct determinants, we 
suggest the following. 

\begin{conj}\label{conj: log saving}
For   an arbitrary finite set  $\cX\subseteq \R$ be of cardinality   $\# \cX = X$,   we have 
\[
D_n(\cX)  \ll  X^{n^2- c \log n}  
\]  
for some absolute constant $c > 0$. 
\end{conj}

\subsection{Further questions} 
It is certainly interesting  to investigate the number $R_{m,n,r}(\cX)$ of $m \times n$-matrices with 
elements from $\cX$ and of rank $r \le m \le n$. It is easy to see the bound
\[
R_{m,n,r}(\cX) \ll X^{mr + (n-r) r} 
\]
(which in fact at least for $m =n$ and  for $r = n-1$ and $r = n-2$ has been implicitly used in the above).

We are also interested in good upper bounds on cardinality of the set 
\[
\#\{\vX\in \cD_n(\cX,d):~\vA + \vX \in \cD_n(\cX,e)\}
\]
for a fixed $n\times n$ matrix $\vA$ and fixed real numbers  $d$ and $e$. 
Having good bounds on this quantity may help to strengthen the argument 
of Section~\ref{sec: large n -1} (some results in this direction 
can be found in~\cite{DKP}). 

The above questions are also interesting for symmetric, orthogonal, symplectic and other special 
matrices. 

Finally, as we have mentioned, matrices over finite fields and  over $\C$ can also be studied via appropriate 
modifications of our arguments, see, for example, Remarks~\ref{rem:Fp 3} and~\ref{rem:Fp n}.

\section*{Acknowledgement}

The  authors would like to  thank  Giorgis Petridis and 
%IS
Misha Rudnev 
for useful discussions. 

During the preparation of this work,   the second author was  supported 
by the Australian Research Council Grants  DP230100530 and DP230100534.


\begin{thebibliography}{99}


\bibitem{AKOS} M. Afifurrahman,   V. Kuperberg, 
A. Ostafe and I. E. Shparlinski, 
`Statistics  of characteristic polynomials and  determinants of rational  matrices',
\textit{Preprint}, 2024,  available  from \url{https://arxiv.org/abs/2401.10086}.

%%\bibitem{AYMRS} E.  Aksoy Yazici, B. Murphy, M.  Rudnev and I.~D.~Shkredov,
%% `Growth estimates in positive characteristic via collisions', 
%%\textit{Intern. Math. Res. Notices}, \textbf{2017} (2017), 7148--7189. 

\bibitem{Arut} 
L. M. Arutyunyan,  `On the growth of the number of determinants with restricted entries',
\textit{Mathem. Notes}, \textbf{109} (2021), 843--848 (translated from 
\textit{Matem. Zametki}). 

 \bibitem{BlLi} V. Blomer and  J. Li, `Correlations of values of random diagonal forms', 
\textit{Intern. Math. Res.  Notices}, \textbf{2023} (2023),  20296--20336. 

\bibitem{BlLu} V. Blomer and C. Lutsko, `Hyperbolic lattice point counting in unbounded rank',
\textit{J. Reine Angew. Math.},  (to appear).
%% 2023, available  from \url{https://arxiv.org/abs/2309.00522}. 

\bibitem{BVW} 
J. Bourgain, V. Vu and P. M. Wood, `On the singularity probability of discrete random 
matrices', \textit{J. Funct. Anal.}, \textbf{258} (2010),   559--603. 

 \bibitem{BreMcK} R. P. Brent and B. D.~McKay, 
`Determinants and ranks of random matrices over $\Z_m$',
\textit{Discrete Math.}, \textbf{66} (1987), 35--49.

\bibitem{BrHBSa}
T. D. Browning, R. Heath-Brown and P. Salberger, 
`Counting rational points on algebraic varieties', 
\textit{Duke Math. J.}, \textbf{132} (2006), 545--578.
 
 
 \bibitem{CJMS} M.  Campos, M. Jenssen, M. Michelen and J. Sahasrabudhe, 
`The singularity probability of a random symmetric matrix is exponentially small',  
\textit{J. Amer. Math. Soc.\/},  (to appear).

\bibitem{DKP} C. Daewoong, D. Koh and T. Pham. 
`Distribution of the determinants of sums of matrices', \textit{Revista Matem. Iber.}, \textbf{37}, (2020), 
1365--1398.


\bibitem{DRS}
W. Duke, Z. Rudnick and P. Sarnak, `Density of integer points on
affine homogeneous varieties', \textit{Duke Math. J.}, \textbf{71} (1993), 143--179.

\bibitem{ET}
G. Elekes and C. D. T\'oth, 
`Incidences of not-too-degenerate hyperplanes', \textit{Proc. 21st ACM  Symp,  Comp. Geometry}, Assoc. Comp. Machinery, 2005,  16-21. 

%\bibitem{IwKow}
% H. Iwaniec and E. Kowalski,
%\textit{Analytic number theory}, Amer.  Math.  Soc.,
%Providence, RI, 2004. 

\bibitem{Kat}  Y. R. Katznelson,
`Singular matrices and a uniform bound for 
congruence groups of ${\SL}_n(Z)$',
\textit{Duke Math. J.}, \textbf{69} (1993), 121--136.

 \bibitem{KSX} O. Klurman, I. D. Shkredov and M.  W. Xu, 
 `On the random Chowla conjecture', 
\textit{Geom. Funct. Anal.}, \textbf{33} (2023), 749--777. 

 \bibitem{KPSV} 
D. Koh,  T. Pham,  C.-Y. Shen and L. A.  Vinh, 
`On the determinants and permanents of matrices with restricted entries over prime fields, 
\textit{Pacific J. Math.}, \textbf{300} (2019),   405--417. 

\bibitem{Lund} 
B. Lund, `Two theorems on point-flat incidences', \textit{Comp, Geometry},  \textbf{92} (2021),  Art.101681.
 
 \bibitem{MOS} A. Mohammadi, A. Ostafe and I. E. Shparlinski,  `On some matrix counting problems',
\textit{Preprint}, 2023, available  from \url{https://arxiv.org/abs/2310.05038}. 

%\bibitem{PetShp}   G. Petridis and I. E. Shparlinski,
%`Bounds on trilinear and quadrilinear exponential sums',
%\textit{J. d'Analyse Math.}, \textbf{138}  (2019), 613--641.

\bibitem{Pila} J. Pila,
`Density of integral and rational points on varieties',
\textit{Ast\'{e}rique}, \textbf{228} (1995), 183--187.

% \bibitem{RaSha} O. E. Raz and M. Sharir, `The number of unit-area triangles in the plane: Theme and
%variatio', 
%\textit{Combinatorica}, \textbf{38} (2017), 1221--1240. 


\bibitem{Rud}  
M. Rudnev, `On the number of incidences between points and planes in three dimensions', 
\textit{Combinatorica}, \textbf{38} (2018),  219--254. 

 \bibitem{RuSh} M.  Rudnev and I. D. Shkredov, `On the growth rate in $\SL_2(\F_p)$,  
 the affine group and sum-product type implications', 
\textit{Mathematika} \textbf{68} (2022), 738--783.

\bibitem{Salb} P. Salberger,  `Counting rational points on projective
varieties', \textit{Proc. London Math. Soc.} \textbf{26}  (2023), 1092--1133.

\bibitem{Shp} I. E. Shparlinski, `Some counting questions for matrices with restricted entries',  
\textit{Lin. Algebra Appl.}, \textbf{432} (2010), 155--160.

 \bibitem{Sol_deZ}  J. Solymosi and F. de Zeeuw, `Incidence bounds for complex algebraic curves on Cartesian
products',  \textit{New Trends in Intuitive Geometry}, Springer, 2018,   385--405. 

\bibitem{SZ}
S. Stevens and  F. de Zeeuw,  `An improved point‐line incidence bound over arbitrary fields',   
\textit{Bull. Lond. Math. Soc.},  \textbf{49} (2017), 842--858.

\bibitem{TaoVu}
T. Tao and V. Vu, \textit{Additive Combinatorics, Cambridge}, Stud. Adv. Math. 105,
Cambridge Univ. Press, Cambridge, 2006.  

 \bibitem{Tikh} K. Tikhomirov, `Singularity of random Bernoulli matrices',  \textit{Annals of Mathematics}, 
 \textbf{191} (2020), 593--634.

 \bibitem{Ver} F. Vermeulen, `Dimension growth for affine varieties', 
\textit{Intern. Math. Res. Notices}, (to appear). 
%%\textbf{2024} (2024), 

 \bibitem{Vinh1} L. A. Vinh, `Distribution of determinant of matrices with restricted entries over finite fields', 
 \textit{J. Combin. Number Theory}, \textbf{1} (2009),  203--212.
 %\textit{Intern. Math. Res. Notices}, (to appear). 
 
 \bibitem{Vinh2} L. A. Vinh, `Singular matrices with restricted rows in vector spaces over finite fields, 
 \textit{Discrete Math.}, \textbf{312} (2012),  413--418.

\bibitem{Vu}  V. Vu, `Recent progress in combinatorial random matrix theory',
\textit{Probability Surveys\/}, {\bf 18} (2021), 179--200. 

 \bibitem{Zahl} J. Zahl, `An improved bound on the number of point-surface incidences in three dimensions', 
\textit{Contrib. Discr. Math.} \textbf{8} (2013), 100--121.

\end{thebibliography}
\end{document}